\documentclass[12pt]{amsart}
\pdfoutput=1

\usepackage{amsmath, amssymb, amsthm, amsfonts, amscd, tikz-cd}

\input xy
\xyoption{all}

\usepackage{hyperref}
\usepackage{graphicx}
\usepackage{color}

\usepackage[margin=1in,marginparwidth=0.8in, marginparsep=0.1in]{geometry}

\usepackage{soul}

\usepackage[all, knot, poly]{xy}
\xyoption{knot}

\usepackage{times}
\usepackage{mathrsfs}

\newtheorem{lemma}{Lemma}
\newtheorem{proposition}[lemma]{Proposition}
\newtheorem{corollary}[lemma]{Corollary}
\newtheorem{theorem}[lemma]{Theorem}

\theoremstyle{definition}
\newtheorem{definition}[lemma]{Definition}

\theoremstyle{remark} 
\newtheorem*{remark}{Remark}

\newtheorem*{example}{Example}

\newcommand{\C}{\mathbb{C}}

\newcommand{\Q}{\mathbb{Q}}
\newcommand{\R}{\mathbb{R}}

\newcommand{\Z}{\mathbb{Z}}

\newcommand{\cO}{\mathcal{O}}

\DeclareMathOperator{\Hom}{Hom}

\DeclareMathOperator{\Fuk}{Fuk}

\DeclareMathSymbol{\shortminus}{\mathbin}{AMSa}{"39}

\begin{document}


\title{Microsheaves from Hitchin fibers via Floer theory}

\author{Vivek Shende}

\begin{abstract}
Fix a non-stacky component of the moduli of stable Higgs bundles, 
on which the Hitchin fibration is proper. 
We show that any smooth Hitchin fiber determines a microsheaf on the global nilpotent cone, 
that distinct fibers give rise to orthogonal microsheaves, and that the endomorphisms
of the microsheaf is isomorphic to the cohomology of the Hitchin fiber.   These results 
are consequences of recent advances in Floer theory.  
Natural constructions on our microsheaves provide plausible candidates for Hecke eigensheaves 
for the geometric Langlands correspondence. 
\end{abstract}

\maketitle

\thispagestyle{empty}

\section{Introduction}

Let $C$ be a smooth compact complex curve, $G$ a reductive group, and $\mathrm{Bun}_G(C)$ the moduli of $G$-bundles on $C$. 
In the context of the geometric Langlands correspondence, one in interested in `Hecke eigensheaves' on $\mathrm{Bun}_G(C)$. 

For geometric reasons (which we recall in Section \ref{sec: hecke}), it is natural to expect a relationship between said eigensheaves and 
fibers of the Hitchin integrable system on $T^* \mathrm{Bun}_G(C)$.  Various
mechanisms have been proposed, including quantization \cite{Beilinson-Drinfeld}, `brane quantization' 
\cite{Kapustin-Witten}, and non-abelian Hodge theory \cite{Donagi-Pantev}.  

\vspace{2mm}

The purpose of the present article is to explain that recent developments in symplectic geometry 
can be used to produce sheaves on $\mathrm{Bun}_G(C)$ directly from Hitchin fibers.  
Let us note right away that, unlike \cite{Beilinson-Drinfeld, Kapustin-Witten, Donagi-Pantev}, we produce
sheaves rather than $D$-modules, i.e. our objects are most directly relevant to the `Betti  geometric Langlands' of \cite{BenZvi-Nadler}.  Additionally, our  sheaves have coefficients, not in $\Q$ or $\C$,
but instead in the Novikov field
$$\Lambda := \bigg\{ \sum_{i = 1}^\infty c_i T^{\lambda_i} \, | c_i \in \Q, \lambda_i \in \R, \, \lambda_i \to \infty \bigg\}$$
Finally, we do not yet know how to prove the resulting objects are 
eigensheaves.  

However, let us also note that, unlike \cite{Beilinson-Drinfeld} and other related approaches, we require no input 
from representation theory of vertex algebras or loop groups etc. 

\vspace{2mm}

In fact, the objects we most directly produce are not sheaves, but rather microsheaves.  Let us recall this
notion, essentially introduced in \cite{Kashiwara-Schapira}.  Let $M$ be a manifold
and $sh(M)$ be a stable category of sheaves.  We recall that given $F \in sh(M)$, 
the microsupport $ss(F) \subset T^*M$ is a locus of directions along which
sections fail to propagate.  Such $ss(F)$ is  (essentially by definition)
preserved by scaling cotangent fibers, and (non obviously) co-isotropic in an appropriate sense.  
The category of locally constant sheaves, $loc(M)$, is characterized by the condition that the microsupport is
contained in the zero section.  
Under the Riemann-Hilbert correspondence, characteristic varieties of regular holonomic D-modules match the microsupports
of their sheaves of solutions. 

For $\Lambda \subset T^*M$, we write $sh_\Lambda(M)$ for the sheaves
whose microsupport is contained in $\Lambda$. 
There is a pre-sheaf of categories on $T^*M$ defined by 
$$\mu sh^{pre}(U) := sh(M) / sh_{T^*M \setminus U}(M)$$ 
We write $\mu sh$ for the corresponding sheaf of categories, 
and refer to elements of $\mu sh(U)$ as microsheaves on $U$.  
It is immediate from the definition that microsheaves depend only on the conic
saturation $\mu sh (U) = \mu sh(\R_{>0} U)$, and not difficult to see that $\mu sh(T^*M) = sh(M)$.  

These notions
extend in the obvious way to the case when $M$ is a smooth stack, like $\mathrm{Bun}_G(C)$.  
Inside $T^*\mathrm{Bun}_G(C)$ we have the conic open locus $\mathrm{Higgs}_G^s(C)$ of stable Higgs bundles, and we may
consider $\mu sh(\mathrm{Higgs}_G^s(C))$.   

Let us explain how such microsheaves relate to the more usual objects of interest.  We denote the locus of stable bundles as
$\mathrm{Bun}_G^s(C)$, and the locus of `very stable' bundles (those not admitting a nilpotent Higgs field) as $\mathrm{Bun}_G^{vs}(C)$.  
We denote the inclusions as
$$\mathrm{Bun}_G^{vs}(C)  \xrightarrow{v} \mathrm{Bun}_G^{s}(C) \xrightarrow{s} \mathrm{Bun}_G(C)$$
We also consider the inclusions
$$T^* \mathrm{Bun}_G^{s}(C) \xrightarrow{\eta} \mathrm{Higgs}_G^s(C) \xrightarrow{\kappa} T^* \mathrm{Bun}_G(C)$$
Finally we denote the Hitchin fibration by $h: T^*\mathrm{Bun}_G(C) \to B$, so $\mathcal{N} := h^{-1}(0) \subset T^* \mathrm{Bun}_G(C)$ is 
the global nilpotent cone of \cite{Laumon}.  When discussing (micro)support conditions, we also 
write $\mathcal{N}$ for the intersection of the nilpotent cone with 
any of the appropriate loci above.  

The aforementioned properties of microsheaf categories imply the following commutative diagram of sheaf and microsheaf restrictions: 

\vspace{2mm}
\noindent
\hspace{-5mm}
\begin{tikzcd}
\mu sh_{\mathcal{N}} (T^*\mathrm{Bun}_G(C)) \arrow[r, "\kappa^*"] \arrow[d, equal] & \mu sh_{\mathcal{N}} (\mathrm{Higgs}_G^s(C)) \arrow[r, "\eta^*"] & \mu sh_{\mathcal{N}} (T^* \mathrm{Bun}_G^{s}(C)) \arrow[d, equal] \arrow[r] 
& \mu sh_{\mathcal{N}}(T^* \mathrm{Bun}_G^{vs}(C)) \arrow[d, equal] \\
sh_{\mathcal{N}} (\mathrm{Bun}_G(C)) \arrow[rr, "s^*"] & & sh_{\mathcal{N}}(\mathrm{Bun}_G^s(C)) \arrow[r, "v^*"] & loc(\mathrm{Bun}_G^{vs}(C)) 
\end{tikzcd} 
\vspace{2mm}

For instance, from the above diagram one can see that an element of $\mu sh_{\mathcal{N}} (\mathrm{Higgs}_G^s(C))$ determines, in particular, 
a local system on  $\mathrm{Bun}_G^{vs}(C)$.

\vspace{2mm}

To avoid having
positive dimensional stabilizers at generic points of $\mathrm{Bun}_G$, we follow  \cite{Donagi-Pantev} in removing the connected component of the generic
stabilizer in the sense of  \cite[Appendix A]{AOV}, and similarly for Higgs bundles (this has no effect when $G$ is semisimple).   
Then $\mathrm{Higgs}_G^s(C)$ is always a smooth orbifold. 
We also recall that  moduli of bundles have connected components
indexed by $d \in \pi_1(G)$; we write $\mathrm{Higgs}_G^s(C)$ for the relevant component. 

For our purposes, we axiomatize the relevant facts about Higgs bundles and Hitchin fibrations into: 

\begin{definition} \label{c h i s}
    By a conic hyperk\"ahler integrable system, we mean a manifold $W$ with hyperk\"ahler complex structures $I,J,K$,  an $I$-holomorphic $\C^*$ action which nontrivially scales $\Omega_I$, and a proper holomorphic map $h: W \to A$, such that the generic fibers of $h$ are Lagrangian for $\Omega_I$ and the $\C^*$-action contracts $A$ to a single fixed point $0$.
\end{definition}

It is well known that $\mathrm{Higgs}_G^s(C)_d$ satisfies Definition \ref{c h i s} so long as (1) there are no strictly semistable Higgs bundles of the given $d$, and (2) there are no orbifold points on $\mathrm{Higgs}_G^s(C)_d$, e.g., when $G = GL_n$ and $d$ is coprime to $n$.  We are not certain how precisely to attribute this well known fact: let us mention the original works of Hitchin \cite{Hitchin, Hitchin-system}, and also further and fuller later developments
\cite{Donaldson, Corlette, Simpson-higgs, Simpson-moduli-1, Simpson-moduli-2, Faltings}.  

We may now state our main result: 

\begin{theorem} \label{intro main theorem} 
Let $h:W \to A$ be a conic hyperk\"ahler integrable system.  Fix any collection of smooth fibers
$L_i := h^{-1}(b_i) \subset W$ for distinct $b_i \in B$.  Then there are  microsheaves 
$F_i \in \mu sh_{h^{-1}(0)}(W)$
such that $F_i$ and $F_j$ are orthogonal for $i \ne j$, and 
$\mathrm{End}(F_i)$ is canonically identified with the cohomology of  $L_i$.  
\end{theorem}

Let us explain the main ideas and issues in the proof of Theorem \ref{intro main theorem}.  By work of Solomon and Verbitsky \cite{Solomon-Verbitsky}, such $L_i$ bound no holomorphic disks, hence define objects of the Fukaya category without further work.  We would like to then apply the Fukaya/sheaf comparison of \cite{GPS3}, but the $L_i$ are necessarily non-exact Lagrangians, and as such \cite{GPS3} does not apply to them directly.  Thus the main task of the present article is to show that one can add non-exact Lagrangians to the Fukaya categories studied in \cite{GPS1, GPS2, GPS3} without essentially enlarging the category.  The basic idea of the proof is that we resolve the diagonal bimodule by resolving the geometric diagonal, which {\em is} an exact Lagrangian, and hence to which the results of \cite{GPS2} apply.

\begin{remark} The requirement that $X$ is a manifold rather than orbifold  appears solely because
the theory of Fukaya categories of symplectic orbifolds is not developed, and in particular we do not 
have the analogue of \cite{GPS3} in this setting.  
\end{remark}


\vspace{2mm}
{\bf Acknowledgements.}  This note began life as an email to David Nadler, in the context of an ongoing discussion of
how to rigorously connect homological mirror symmetry with (Betti) geometric Langlands. 
I also thank Dima Arinkin, Ron Donagi, Dennis Gaitsgory, Tony Pantev, and John Pardon for helpful conversations.

This work was supported by
 Villum Fonden Villum Investigator grant 37814, Novo Nordisk Foundation grant NNF20OC0066298, and Danish National Research Foundation grant DNRF157, and NSF CAREER DMS-1654545.

\section{Hitchin fibers and Hecke eigensheaves} \label{sec: hecke}

Let us recall what Hecke eigensheaves are, and why one might hope to get them from fibers of the Hitchin system \cite{Beilinson-Drinfeld, Kapustin-Witten,
Donagi-Pantev}.  This section is purely motivational and logically unrelated to the results presented
in the remainder of the article. 

For a space $X$, we write $[X]$ for some appropriate category of $D$-modules or constructible sheaves on $X$. 
There are `Hecke operators' $H_\mu: [\mathrm{Bun}_G(C)] \to [\mathrm{Bun}_G(C) \times C]$, given by
convolution with respect to a correspondence parameterizing pairs of bundles which are isomorphic away from a single 
point $c \in C$, and whose difference at this point is controlled by a cocharacter $\mu$ of $G$.  

The geometric
Langlands conjecture asserts, in particular, the existence of certain $F \in [\mathrm{Bun}_G(C)]$ which are `Hecke eigensheaves' 
in the sense that 
$H_\mu ( F ) = F \, \boxtimes \, \rho^\mu(\chi)$, where $\chi$ is a $G^\vee$ local system on $C$, and $\rho^\mu$ is the 
representation corresponding to the character $\mu$ of $G^\vee$.  

Suppose given a functor $|\,\cdot\,|: [X] \to \mathrm{Subsets}(T^*X)$ carrying composition of integral transforms to composition of subset relations.  Such functors in the case of D-modules and constructible sheaves are given by characteristic cycle and microsupport, respectively. 
For these cases, one knows that 
one connected component of $|H_\mu| \subset T^* \mathrm{Bun}_G(C) \times T^* \mathrm{Bun}_G(C) \times T^* C$ is the closure of the conormal to the smooth locus of the Hecke correspondence; let us denote it $N_\mu$. 

One checks readily that 
$N_\mu \circ h^{-1}(b) = h^{-1}(b) \times \widetilde{C}_{b}^\mu$.  Here $\widetilde{C}_{b}^\mu \subset T^*C$ is the spectral cover
corresponding to the point $b \in B$ and representation $\mu$.  That is, the Hitchin fibers are `$N_\mu$-eigensets'.
Plausibly (one could check numerically), 
in fact $N_\mu$ is the only component of $|H_\mu|$ which acts nontrivially on a generic Hitchin fiber, and the fibers
are in fact $|H_\mu|$-eigensets.  

Now, for $D$-modules and constructible sheaves, the image of $|\, \cdot \,|$ consists only of conic subsets of the cotangent bundle, and in particular never $h^{-1}(b)$.  One could however imagine that there is some other setup $[\, \cdot \,], |\, \cdot \,|$ in which nonconic subsets do appear, but which admits a comparison to a theory of D-modules or constructible sheaves.  
For instance, one can imagine that the approach of 
Beilinson and Drinfeld -- ``quantizing the Hitchin system'' -- has some (perhaps known) interpretation along these lines where  $[\, \cdot \,]$ is  quantization of the cotangent bundle with some and $|\, \cdot \,|$ is taking the support in some semiclassical limit.  Ideas along these lines
 include \cite{Arinkin, Donagi-Pantev-classical, Bezrukavnikov-Braverman}.  Another such setup is nonabelian Hodge theory \cite{Donagi-Pantev, Donagi-Pantev-parabolic}. 

 \vspace{2mm}


A different way to access objects of non-conic origin is through the relationship between sheaves and Fukaya categories, as proposed in \cite{Kapustin-Witten}. 
One can regard this note as explaining what precisely can be done in this direction with existing technology.

\section{Fukaya categories and microsheaves} 

Our basic tool for producing microsheaves from Lagrangians is the recent comparison theorem resulting
from the works \cite{GPS1, Shende, GPS2, GPS3, Nadler-Shende}. 

\begin{theorem} \label{GPS3 comparison} \cite{GPS3} 
Let $W$ be a Weinstein symplectic manifold.  
We write $\mathrm{Fuk}(W)$ for the wrapped Fukaya category.  
Then there is an equivalence $\mathrm{Ind}(\mathrm{Fuk}(W)) \cong 
\mu sh(\mathrm{Core}(W))$.\footnote{Or anti-equivalence, depending on a number of universal sign conventions; see \cite{GPS3}.} 
\end{theorem} 

In this section we give a brief review of some of the notions involved on the `Fukaya' side. 

\vspace{2mm} 

Given a symplectic manifold $(W, \omega)$, it is often possible to form an 
$A_\infty$-category whose objects are Lagrangian submanifolds of $W$, whose
morphism spaces $\Hom(L, L')$ are generated by intersection points of (appropriate perturbations of)
$L$ and $L'$, and whose structure maps -- differential, composition, higher compositions -- are 
defined by counting holomorphic curves with boundary along the Lagrangians.  
Foundational references include \cite{FOOO, Seidel}.  

We restrict attention to symplectic manifolds equipped with a Liouville structure: a vector field $Z$ with $Z \omega = \omega$, 
such that $Z$ identifies the complement of a compact set in $W$ with the positive symplectization of a contact manifold. 
Note $d(i_Z \omega) = \omega$, and $W$ is noncompact.  The locus of points  which remain bounded under the flow of $Z$
is termed the `spine' or `skeleton' or `core'. 

The prototypical example is a cotangent bundle $T^*M$ with the vector field $Z$ generating the radial dilation; 
for the compact subset take the co-disk bundle; the core is the zero section. 

To any Liouville $W$ is associated the `wrapped' Fukaya category $\mathrm{Fuk}(W)$ \cite{Abouzaid-Seidel, GPS1}. 
The basic objects of $\mathrm{Fuk}(W)$ are exact Lagrangian submanifolds, 
which are $Z$-conic outside a compact set.\footnote{
More precisely, the definition of $\Fuk(W)$ involves the choices of certain topological structures on $W$ and on the Lagrangians 
$L$.  In the present article, $W$ will always be Calabi-Yau, and we only consider oriented, spinned, graded Lagrangians.  
In this case we may define $\Fuk(W)$ over
any ring $k$, and equip the morphism spaces with a $\Z$-grading.}  
The term
`wrapped' refers to the fact that 
trajectories at infinity are incorporated into the morphism spaces, which are typically infinite dimensional when noncompact
Lagrangians are involved.  
We write $\mathrm{Fuk}(W)$ to mean the triangulated category generated by such objects and morphisms.

\begin{example}
$\mathrm{Fuk}(T^* S^1)$ is equivalent to the category of perfect dg modules for $k[t, t^{-1}]$. 
More generally, $\mathrm{Fuk}(T^* M)$ is equivalent to the category of perfect dg modules for the algebra of chains
on the based loop space of $M$ \cite{Abouzaid-cotangent}. 
\end{example} 

When $Z$ is the gradient flow of some Bott-Morse function, the Liouville structure is said to be Weinstein.  
In this case, the core is a finite union of locally closed isotropic submanifolds.  The cotangent bundle above is an example; 
the distance to the zero section gives the Bott-Morse function.  Stein complex manifolds (with the plurisubharmonic witness
providing the Morse function) yield more examples.  For a detailed study of these manifolds, see \cite{Cieliebak-Eliashberg}. 


%
%

\vspace{2mm}
The works \cite{GPS1, GPS2, GPS3} were written with coefficients $\mathbb{Z}$, as is common when working in exact symplectic geometry.  Let us recall that more generally (in non-exact contexts) it is necessary to work over the Novikov field $\Lambda$, and count each holomorphic curve $C$ formally weighted by $T^{\mathrm{Area}(C)}$.   
The reason why the Novikov field is needed in general is that the Gromov compactness theorem only holds given an a priori bound on the area of holomorphic curves, hence one counts by $T^{\mathrm{Area}}$ to avoid infinite sums.  But in the exact setting, these sums are always finite, since the relative homology class of the curve determines the area, by Stokes' theorem.  So one can, and many authors do, immediately set $T=1$.  Said differently, in the exact case, one can work over a polynomial version of the Novikov field, which one has at the end the freedom to either set $T=1$ or formally complete in $T$ to pass to the Novikov field. 
The arguments of \cite{GPS1, GPS2, GPS3} apply without change in this setup (of course still restricting, for now, to exact Lagrangians).   So as to subsequently adapt the results of those works to include some non-exact Lagrangians, we will work henceforth over the Novikov field.

\section{Nonexact lagrangians} 

The comparison result of Theorem \ref{GPS3 comparison} cannot be immediately applied to prove 
Theorem \ref{intro main theorem} because the wrapped Fukaya category used in \cite{GPS1, GPS2, GPS3} 
only allows {\em exact} Lagrangians, and Hitchin fibers $h^{-1}(b)$ are not exact unless $b=0$.\footnote{The nilpotent cone 
$\mathcal{N} = h^{-1}(0)$ is exact, and its smooth components do provide objects of the wrapped Fukaya category, 
and support corresponding microsheaves.  The existence of such microsheaves is elementary, 
and does not require the comparison theorem.}  The purpose of the present section is to extend the comparison theorem
to incorporate nonexact Lagrangians.

Let us first recall that Floer homology between nonexact Lagrangians is well defined if all Lagrangians involved are 
either (1) `tautologically unobstructed'
i.e. bound no holomorphic disks, or more generally (2) are equipped  with `bounding cochains' in the sense of \cite{FOOO}. 
However, in either case, one must work with coefficients in the Novikov field $\Lambda$.

Let us write $\mathrm{Fuk}^+(W)$ for  
the category defined in \cite[Remark 2.32]{GPS2}: objects are pairs $(L, J)$ where $L$ is an eventually conic (but not necessarily exact) Lagrangian,
and $J$ is an eventually conic compatible with the symplectic form $\omega$, such that 
$L$ bounds no holomorphic disks.\footnote{This does not require the general foundations of \cite{FOOO}.  Using \cite{FOOO}, 
we could more generally allow Lagrangians which do bound holomorphic disks but are equipped with boundary cochains.}

It is obvious from the definitions that $\mathrm{Fuk}(W) \subset  \mathrm{Fuk}^+(W)$.  To apply Theorem \ref{GPS3 comparison}
to nonexact Lagrangians, it will suffice to show that this inclusion is in fact an equivalence.  Doing so is the purpose of the remainder 
of the present section, and the essential new technical contribution of the present article.

\vspace{2mm}
 
We recall that \cite[Thm. 1.5; Sec. 8]{GPS2} gives
a K\"unneth embedding: $\mathrm{Fuk}(V) \otimes \mathrm{Fuk}(W) \hookrightarrow \mathrm{Fuk}(V \times W)$. 
Two important observations needed to establish this result are: first, that `product' wrappings in $V \times W$ 
are cofinal in all wrappings, and second, that while the product of eventually conic but not conic Lagrangians
is not eventually conic, one can straighten out the product to achieve this.  

Rather than contemplate straightening
products where one factor is nonexact, we will make do with the following weaker result: 

\begin{proposition}\label{prop: weak Kunneth}
There is a morphism $\kappa: \mathrm{Fuk}(V \times W)^{op} \to \mathrm{Mod}(\mathrm{Fuk}^+(V) \otimes \mathrm{Fuk}^+(W))$, with the following
properties: 
\begin{enumerate}
\item \label{conics} If $L \subset V$ and $M \subset W$ are both 
(not just eventually) conic, the module $\kappa(L \times M)$ is represented by $L \otimes M$. 
\item \label{diagonal} When $\Delta \subset (-W) \times W$ is the diagonal, then 
$\kappa(\Delta) \subset \mathrm{Mod}(\mathrm{Fuk}^+(-W) \otimes \mathrm{Fuk}^+(W)) = 
\mathrm{Mod}(\mathrm{Fuk}^+(W)^{op} \otimes \mathrm{Fuk}^+(W))$ is the diagonal bimodule.   
\end{enumerate} 
\end{proposition} 
\begin{proof}
The construction of $\kappa$ proceeds exactly as in \cite[Sec. 8.2]{GPS2}; we give the highlights.  
Recall that the wrapped Fukaya category $\mathrm{Fuk}(W)$  is by definition a localization 
of a `directed' category $\cO(W)$.   As in \cite[Sec. 8.2]{GPS2}, one counts appropriate holomorphic quilts to define 
a natural $A_\infty$ functor to chain complexes, $$\cO(V \times W)^{op} \otimes \cO^+(V) \otimes \cO^+(W) \to \mathrm{Ch}$$
In this formulation the necessary Gromov compactness is deduced from the eventual conicity of Lagrangians in each
factor separately, not the (false) conicity of product Lagrangians.  
Such a map is (by definition) a morphism 
$\cO(V \times W)^{op} \to \mathrm{Mod}(\cO^+(V) \otimes \cO^+(W))$.  As in \cite[Sec. 8.2]{GPS2}, 
one sees that this descends along the localization to wrapped categories by using the fact that this localization on $\cO(V \times W)$
can be implemented using only product wrappings on $V \times W$.  

The assertions \eqref{conics}, \eqref{diagonal} are essentially tautologies in the definitions of the holomorphic quilts involved.  
\end{proof} 

The way in which Proposition \ref{prop: weak Kunneth} is weaker (and the reason it is easier to prove) 
than the K\"unneth result of \cite[Thm. 1.5]{GPS2}
is that it does not assert a priori that $\kappa$ lands in the representable modules.  
Establishing this representability in \cite{GPS2} 
required considering straightening of products of Lagrangians, which we do not know how to do when one factor is noncompact and the other is nonexact.  
Nevertheless: 

\begin{proposition}
When $V, W$ are Weinstein, the image of $\kappa$ consists of representable bimodules.
\end{proposition} 
\begin{proof}
Proposition \ref{prop: weak Kunneth} \eqref{conics} asserts that the images under $\kappa$ of products of conic elements 
are representable.  
By Theorem \cite[1.13]{GPS2}, the category $\Fuk(V \times W)$ is generated by cocores, which are themselves products of 
the (conic) cocores of $V$ and $W$ (see \cite[Sec. 9.1]{GPS2}).  
\end{proof} 

\begin{remark}
Note this does not give a new argument for \cite[Thm. 1.5]{GPS2}, since  \cite[Thm. 1.13]{GPS2} is proved using \cite[Thm. 1.5]{GPS2}. 
\end{remark}

\begin{theorem} \label{thm: nothing new} 
For a Weinstein manifold  $W$, the inclusion $\mathrm{Fuk}(W) \subset \mathrm{Fuk}^+(W)$ is an equivalence. 
\end{theorem}
\begin{proof}
Let us first recall some tautological algebraic facts.  Consider a dg or $A_\infty$ category  $A$, with diagonal bimodule
$\Delta \in (A^{op} \otimes A)-\mathrm{mod}$, i.e. $\Hom_A(a, b) \cong \Delta(a, b)$.  Note that 
$b \mapsto \Delta(\cdot, b) \in A^{op}-\mathrm{mod}$ is the Yoneda embedding.  
Suppose $\Delta$ is representable by an object in $A^{op} \otimes A$ which is expressable as a 
complex of objects with terms $a_i \otimes b_i$.  Then $\Delta(\cdot, b)$ is expressable as 
a complex of objects with terms $\Hom_{A^{op}}(a_i, \cdot) \otimes \Hom_A(b_i, b)$.   That is, the $a_i$ generate $A$. 
(This is the categorification of the statement that if $V$ is a vector space and $1_V = \sum a_i \otimes b_i$, then
the $a_i$ must span $V$.) 

Now Proposition \ref{prop: weak Kunneth} \eqref{diagonal} asserts that the 
diagonal bimodule of $\mathrm{Fuk}^+(W)$ is the image of the geometric diagonal
in $W^- \times W$.  The diagonal is conic and exact, so already an element of $\mathrm{Fuk}(W^- \times W)$, where 
it is generated by cocores of $W^- \times W$ by  \cite[Theorem 1.13]{GPS2}.  These cocores are themselves 
products of the cocores of $W^-$ and $W$ (see \cite[Sec. 9.1]{GPS2}).   Thus, the cocores (which were already elements
of $\Fuk(W)$) generate $\Fuk^+(W)$. 
\end{proof}

\begin{corollary} \label{cor: novikov GPS3} 
Over the Novikov field,
$\mathrm{Ind}(\mathrm{Fuk}^+(W))  \cong \mu sh (\mathrm{Core}((W))$. $\square$
\end{corollary}

\section{Floer theory in hyperk\"ahler manifolds} \label{sec: hk}

Let $W$ be a manifold.  We recall that a `hypercomplex' structure on $W$ is an
action of the quaternions on $TW$ such that $I, J, K$ determine integrable complex structures.  
A metric on a hypercomplex manifold is said to be hyperk\"ahler if it is K\"ahler for each complex structure separately.  
We denote the corresponding symplectic forms $\omega_I, \omega_J, \omega_K$.  
Note that $\Omega_I := \omega_J + i \omega_K$ gives a holomorphic symplectic form in complex structure $I$, etcetera. 

A submanifold which is holomorphic for all complex structures would be itself hypercomplex; in particular, 
of dimension divisible by 4.  Thus holomorphic curves cannot remain holomorphic under perturbation of the complex structure 
within the unit quaternion family.  A strong form of this fact will play a crucial role for us here: 

\begin{theorem} \label{thm: solomon-verbitsky} \cite{Solomon-Verbitsky} 
Let $W$ be a complete hyperk\"ahler manifold, and $L \subset W$ a holomorphic Lagrangian
for $(I, \Omega_I)$.   

Then for all but countably many complex structures of the form $J_\xi:= J \cos(\xi) + K \sin(\xi)$, 
there are no $J_\xi$-holomorphic curves with boundary along the $L$.   

When $J_\xi$ is moreover tamed by $\omega_J$, 
there is an $A_\infty$ equivalence $\mathrm{HF}_{\omega_J}(L, L) \cong \mathrm{H}^*(L)$; in particular, the LHS is formal. 
\end{theorem}

\begin{remark}
Throughout, $\omega_J$ may be replaced by any fixed $\omega_\theta := \mathrm{Re}(e^{\i \theta} \Omega_I) = 
\omega_J \cos(\theta) + i \omega_K \sin(\theta)$.  
\end{remark}

When $(W, \omega_J)$ is Liouville, by Theorem \ref{thm: solomon-verbitsky} we obtain objects
$(L, J_\xi) \in \Fuk^+(W, \omega_J)$. 

\begin{corollary} \label{cor: hkw}
Let $W$ be a complete hyperk\"ahler manifold such that $(W, \omega_J)$ is Weinstein.
Assume in addition that $L_1, \ldots, L_n$ are compact spinned $I$-holomorphic 
$\Omega_I$-Lagrangians.\footnote{Note that by virtue of holomorphicity, the $L_i$
are canonically oriented.  Also recall that a hyperk\"ahler manifold carries a Calabi-Yau form, with respect to which
holomorphic Lagrangians are special and hence admit gradings.} 
Then there are (a nonempty open set of) $\xi_i \in S^1$ such that $(L_i, J_{\xi_i})$ define objects of $\mathrm{Fuk}^+(W, \omega_J)$.
There are corresponding microsheaves $F_1, \ldots, F_n \in \mu sh(\mathrm{Core}(W))$ and isomorphisms
$$\Hom_{\mathrm{Fuk}^+(W)}((L_i, J_{\xi_i}), (L_j, J_{\xi_j}) \cong \Hom_{\mu sh(\mathrm{Core}(W))}(F_i, F_j)$$
The same holds if we equip the $L_i$ with ($\Lambda$-valued) unitary local systems.  So long as these local systems
are rank one, we have also an $A_\infty$ equivalence $\Hom_{\mathrm{Fuk}(W)}(L_i, L_i) \cong \mathrm{H}^*(L_i)$. 
\end{corollary}
\begin{proof}
Follows immediately from Corollary \ref{cor: novikov GPS3} and Theorem \ref{thm: solomon-verbitsky}.
\end{proof}

Let us recall a typical situation giving rise to hyperk\"ahler and Weinstein manifolds.

\begin{definition}
By a {\em complex Liouville manifold}, we mean a holomorphic symplectic manifold 
$(W, I, \Omega)$ carrying an $I$-holomorphic 
$\C^*$ action, which acts by a positive character on the line $\C \Omega$,   
such that for any $w \in W$, the limit over $z \in \C^*$ 
given by $\lim_{z \to 0} z w$ always exists.
\end{definition} 

\begin{lemma}
If $(W, I, \Omega)$ is complex Liouville, then each  real symplectic manifolds $(W, \mathrm{Re}(e^{i \theta} \Omega) )$
is Liouville.
\end{lemma} 
\begin{proof}
For the Liouville vector field, take a rescaled generator of the $\R^{> 0} \subset \C^*$ action.
\end{proof}

\begin{proposition} \label{hyperk weinstein}
Assume a complex Liouville manifold $(W, I, \Omega)$
extends to a complete hyperk\"ahler structure such that the $S^1 \subset \C^*$ action is Hamiltonian for $\omega_I$.  Then $(W, \mathrm{Re}(e^{i \theta} \Omega))$ is Weinstein.  
\end{proposition} 
\begin{proof}
The Liouville vector field is the gradient with respect to the hyperk\"ahler metric of the moment map for the $S^1$ action.  
\end{proof}

\begin{proof}[Proof of Theorem \ref{intro main theorem}]
Definition \ref{c h i s} translated through Proposition \ref{hyperk weinstein} allows us to invoke Corollary \ref{cor: hkw}.  (Connected components of fibers of an integrable system are tori, hence carry a canonical spin structure.) 
Finally, the hypothesis that $h:W\to A$ is proper and $\C^*$-equivariant for an action which contracts $A$ to a point makes clear that $h^{-1}(0)$ is the locus in $W$ which does not escape under the Liouville flow, i.e. its core. 
\end{proof}

\begin{remark} 
Returning to the setting of Higgs bundles, 
let $L$ be the object of the Fukaya category and $F$ be the microsheaf associated to some 
smooth Hitchin fibre $h^{-1}(b)$.
We can determine the rank of the local system obtained by further restricting $F$ to the locus of very-stable bundles.  Let  $T$ be the cotangent fiber at a very stable bundle.  Then 
by \cite{GPS3}, $\Hom(T, L)$ 
is the stalk of $F$.  
For a generic choice of cotangent fiber, the intersection $T \cap h^{-1}(b)$ is transverse.  The indices
of the intersections of complex Lagrangians are all equal,
so $\Hom(T, L)$ is concentrated in one degree, and has rank equal to the intersection
number of $T$ and $h^{-1}(b)$, which in turn is equal to the (positive) multiplicity of the locus of stable
bundles in the (nonreduced) nilpotent cone.  For $\mathrm{GL}_n$, this number is computed
explicitly in \cite{Hausel-Hitchin}.  

It may be interesting to explore the microstalks on other components, for which the slices
constructed in \cite{Collier-Wentworth, Hausel-Hitchin} should be useful. 
\end{remark}

\begin{remark} \label{rem: GMN}
The cotangent
bundle of a complex manifold is complex Liouville, but in general only carries an (incomplete) hyperk\"ahler metric 
in a neighborhood of the zero section \cite{Feix, Kaledin-cotangent}.  We expect that nevertheless some version of Theorem \ref{thm: solomon-verbitsky},
and hence Corollary \ref{cor: hkw}, should be true in this context.  Specialized to the case of cotangent bundles of curves, 
such a result would imply the existence of a Floer theoretic functor carrying local systems on spectral covers
to local systems on the base.  It is natural to expect that the trees which calculate this functor in the flow-tree limit \cite{Ekholm}
are given by the `spectral network' prescription of \cite{GMN}.  
\end{remark}

\begin{remark} \label{rem: multiplicative}
There are natural examples which are (compatibly) Weinstein and complex symplectic, but not complex 
Liouville or complete hyperk\"ahler, such as a neighborhood of a nodal elliptic curve in an elliptically fibered K3 surface. 
More general examples of such spaces appear in
\cite{Dancso-McBreen-Shende, Gammage-McBreen-Webster, Aganagic}.  It is desirable
to generalize Theorem \ref{thm: solomon-verbitsky} to this context as well.  
\end{remark}

\begin{remark}
The `brane quantization' picture of \cite{Gukov-Witten, Gaiotto-Witten} -- applied to the problem of 
geometric Langlands in \cite{Kapustin-Witten} -- would, if made rigorous, likely relate our
constructions to deformation and/or geometric quantization.  There is long-ongoing work on formulating
precise mathematical conjectures about the relation between quantization and Floer theory
in holomorphic settings \cite{Kontsevich, Soibelman}.  See \cite{Kuwagaki} for another point of view on this
conjectural correspondence in the special case of cotangent bundles of curves.  
\end{remark}

\section{Discussion}

Let us comment on a few difficulties to be overcome to bring our microsheaves
into some precise relation with objects
of interest to Langlands geometers. 

One problem is that we are forced to work over the Novikov field $\Lambda$, whereas geometric Langlands is typically formulated  over $\mathbb{C}$.  
As usual in mirror symmetry \cite{Kontsevich-Soibelman}, one should either prove convergence of holomorphic disk counts, or
formulate some version of the geometric Langlands correspondence appropriate to the nonarchimedean coefficients $\Lambda$. 


Another pressing issue is the fact that, 
in order to work in a context where Floer theory is well defined, we restrict to the stable locus and assume there are no corresponding
semistables.   However, 
the locus of stable Higgs bundles is not preserved by the Hecke correspondences.  One can
microlocalize and restrict the Hecke correspondences to the stable locus, but it is not clear how the resulting algebra action is related to the original. 
Additionally, the discrete data $d$ -- which we fixed to avoid strictly semistable Higgs bundles -- is not preserved by the 
Hecke correspondences.  One could restrict to the part of the Hecke algebra which preserves $d$, but would have
to provide geometric representatives for generators of this algebra.  

More ambitiously, one could try and work with or around the symplectic singularities; see 
\cite{Manolescu-Woodward, Daemi-Fukaya} for some ideas in this direction in a similar setting.  

Other difficulties arise because Floer theory is, at present, only defined for smooth Lagrangians. 
In particular, without further constructions, we are restricted to studying smooth Hitchin fibers, and, to check the
Hecke eigen-ness
Floer theoretically, we would  need to provide smooth representatives for Hecke algebra elements.  
We hope that further developments around Lagrangian skeleta may help in this regard.

\bibliographystyle{plain}
\bibliography{refs}
\end{document}